\documentclass{amsart}
\usepackage{latexsym,amssymb,amsmath,amscd,amscd,amsthm,amsxtra}
\usepackage{amsfonts}
\usepackage{graphicx}

\newtheorem{theorem}{Theorem}[section]
\newtheorem{corollary}[theorem]{Corollary}
\newtheorem{lemma}[theorem]{Lemma}
\newtheorem{proposition}[theorem]{Proposition}
\theoremstyle{definition}
\newtheorem{definition}{Definition}[section]
\theoremstyle{remark}
\newtheorem{remark}{Remark}[section]
\numberwithin{equation}{section}

\begin{document}

\title{Holomorphic harmonic analysis on complex reductive groups}

\author{Jinpeng An}
\address{School of mathematical sciences, Peking University,
 Beijing, 100871, P. R. China. Current address: Department of Pure Mathematics, University of
Waterloo Waterloo, Ontario N2L 3G1, Canada}
\email{j11an@math.uwaterloo.ca}

\author{Zhengdong Wang}
\address{School of mathematical sciences, Peking University,
 Beijing, 100871, P. R. China}
\email{zdwang@pku.edu.cn}

\author{Min Qian}
\address{School of mathematical sciences, Peking University,
 Beijing, 100871, P. R. China}

\thanks{This work is supported by the 973 Project
Foundation of China (\#TG1999075102).}

\begin{abstract}
We define the holomorphic Fourier transform of holomorphic functions
on complex reductive groups, prove some properties like the Fourier
inversion formula, and give some applications. The definition of the
holomorphic Fourier transform makes use of the notion of
$K$-admissible measures. We prove that $K$-admissible measures are
abundant, and the definition of holomorphic Fourier transform is
independent of the choice of $K$-admissible measures.
\end{abstract}



\maketitle

\section{Introduction}

Let $f$ be a holomorphic function on $\mathbb{C}\setminus\{0\}$. It
is a standard result in complex analysis that the Laurant series
$$f(z)=\sum_{n=-\infty}^{+\infty}a_nz^n$$ of $f$ converges locally uniformly on
$\mathbb{C}\setminus\{0\}$. From the viewpoint of representation
theory, $z\mapsto z^n$ $(n\in\mathbb{Z})$ are the holomorphic
representations of the multiplicative group
$\mathbb{C}\setminus\{0\}$. In this paper, we generalize this fact
to complex reductive groups.

Let $G$ be a complex reductive group. Denote the space of
holomorphic functions on $G$ by $\mathcal{H}(G)$. Our main goal is
to expand any $f\in\mathcal{H}(G)$ as a holomorphic Fourier series
$$f=\sum_{\pi,i,j}\lambda_{\pi,ij}\pi_{ij},$$ which converges locally
uniformly on $G$, where $\pi$ runs over all holomorphic
representations of $G$, $\pi_{ij}$ are the matrix elements of $\pi$.
Note that $\mathcal{H}(G)$ can be endowed with a structure of
Fr\'echet space for which convergence is equivalent to locally
uniform convergence, the Fourier expansion implies that the subspace
$\mathcal{E}$ of $\mathcal{H}(G)$ consists of linear combinations of
matrix elements of holomorphic representations of $G$ is dense in
$\mathcal{H}(G)$. We then prove that the holomorphic Fourier
expansion satisfies the usual properties of Fourier expansion like
the Fourier inversion formula and the Plancherel Theorem. We also
provide applications of the holomorphic Fourier expansion to
holomorphic class functions and holomorphic evolution partial
differential equations on $G$.

To do this, we select a class of auxiliary measures on $G$, which
are called \emph{$K$-admissible measures} in this paper. A measure
$d\mu$ on $G$ is $K$-admissible if\\
(1) $d\mu$ is of the form $d\mu(g)=\mu(g)dg$, where $dg$ is a (left)
Haar measure, $\mu$ is a measurable function on $G$ locally bounded
from below;\\
(2) all holomorphic representations of $G$ are
$L^2$-$d\mu$-integrable; and\\
(3) $d\mu$ is $K$-bi-invariant, where $K$ is a chosen maximal
compact subgroup of $G$.\\
Condition (1) implies that an $L^2$-$d\mu$-convergent sequence of
holomorphic functions is locally uniformly convergent. So we can
prove the locally uniform convergence of a sequence of holomorphic
functions by showing that it is $L^2$-$d\mu$-convergent, which is
usually easier to handle. Condition (2) ensures that the matrix
elements of holomorphic representations are $L^2$-$d\mu$-integrable.
So we can first expand functions as Fourier series in the $L^2$
sense, and get the locally uniform convergence by Condition (1).
Condition (3) enable us to use representation theory of the compact
group $K$. We will show that $K$-admissible measures on $G$ are
abundant enough. In particular, for any $f\in\mathcal{H}(G)$, there
exist a $K$-admissible measure $d\mu$ such that $f$ is
$L^2$-$d\mu$-integrable.

The main idea to expand holomorphic functions on $G$ as Fourier
series is as follows. We first prove two theorems of Peter-Weyl-type
on $G$, which, among other things, provide an orthogonal basis of
the Hilbert space of $L^2$-$d\mu$-integrable holomorphic functions,
consisting of the matrix elements of holomorphic representations of
$G$. Then, for a holomorphic function $f$ on $G$ that is to be
expanded, we choose a $K$-admissible measure $d\mu$ such that $f$ is
$L^2$-$d\mu$-integrable. By the Peter-Weyl-type theorem, $f$ can be
expanded as a Fourier series in the $L^2$ sense. Since
$L^2$-$d\mu$-convergence implies locally uniform convergence, we get
the locally uniformly convergent Fourier expansion of $f$. Then we
prove that such expansion is independent of the choice of the
$K$-admissible measure $d\mu$.

Now we review the contents of the following sections more closely.
After a brief recollection of properties of complex reductive groups
in Section 2, we will define the notion of $K$-admissible measures
in Section 3, and prove that they are abundant. Two theorems of
Peter-Weyl-type, that is, the $L^2$ case and the locally uniform
case, will be proved in Section 4. The $L^2$ Peter-Weyl-type theorem
is due to Hall \cite{Ha}. In his proof of the completeness part of
the theorem, Hall used some analytical techniques like
change-of-variables on $G$, Laplace operators, and the Monotone
Convergence Theorem. Our proof of the completeness part will be
group-representation-theoretic. We will make use of the complete
reducibility of representations of compact groups. The holomorphic
Fourier transform will be studied in Section 5. The Fourier
inversion formula and the Plancherel Theorem will be proved. Some
basic properties of the holomorphic Fourier transform will also be
given. Section 6 will be devoted to applications of the holomorphic
Fourier transform to holomorphic class functions and holomorphic
evolution partial differential equations on complex reductive
groups.

The first author would like to thank Professor Xufeng Liu for
helpful discussion, and for his successive help and encouragement.


\section{Preliminaries on complex reductive groups}

The notion of ``reductive groups'' has different definitions by
different authors. For technical reason, we adopt Hochschild
\cite{Ho}. We briefly recall some properties of complex reductive
groups below, which are to be used in the following sections. The
detailed proofs can be found in \cite{Ho} or other textbooks on Lie
groups. For simplicity, all Lie groups in this paper are assumed to
be connected, although many assertions also hold if ``connected'' is
replaced by ``with finitely many connected components''.

\begin{definition}\label{D:reductive}(Hochschild
\cite{Ho}) A complex Lie group $G$ is \emph{reductive} if $G$ admits
a faithful finite-dimensional holomorphic representation and every
finite-dimensional holomorphic representation of $G$ is completely
reducible.
\end{definition}

Every complex semisimple Lie group is reductive.
$\mathbb{C}^{*}=\mathbb{C}\backslash\{0\}$, as a multiplicative
group, is reductive. But $\mathbb{C}$ and
$\mathbb{C}/\mathbb{Z}^{2}$, with the canonical holomorphic
structures, are not reductive. In fact, we have

\begin{proposition}\label{P:structure}
A complex Lie group $G$ is reductive if and only if $G$ has the form
$(H\times(\mathbb{C}^*)^n)/\Gamma$, where $H$ is a simply connected
complex semisimple Lie group, $\Gamma$ is a finite central subgroup
of $H\times(\mathbb{C}^*)^n$.
\end{proposition}

The following two propositions concern the relations of complex
reductive groups and their maximal compact subgroups. For a compact
Lie group $K$, we denote its complexification by $K_{\mathbb{C}}$.

\begin{proposition}\label{T:correspondence}
The complexification of a compact Lie group is reductive.
Conversely, all maximal compact subgroups of a complex reductive
group is connected, and any two of them are conjugate. Moreover, if
$K$ is a maximal compact subgroup of a complex reductive group $G$,
then $G\cong K_{\mathbb{C}}$.
\end{proposition}

\begin{proposition}\label{T:representation}
Let $G$ be a complex reductive group, $K$ a maximal compact subgroup
of it. Then every finite-dimensional unitary representation
$\sigma:K\to U(n)$ of $K$ can be uniquely extended to a holomorphic
representation $\sigma_{\mathcal{H}}:G\to GL(n,\mathbb{C})$, and
$\sigma$ is irreducible if and only if $\sigma_{\mathcal{H}}$ is
irreducible.
\end{proposition}

We denote by $\widehat{K}$ the unitary dual of $K$, and by
$\widehat{G}_{\mathcal{H}}$ the set of equivalence classes of all
finite-dimensional holomorphic irreducible representations of $G$.
By the above proposition, we have a bijection
$\widehat{G}_{\mathcal{H}}\leftrightarrow\widehat{K}
\;([\pi]\leftrightarrow[\pi|_K])$, where $[\cdot]$ denote the
equivalence class of a representation.

The next proposition will be useful in the following sections.

\begin{proposition}\label{T:extension}(\cite{Va}, Lemma 4.11.13)
Let $G$ be a complex reductive group, $K$ a maximal compact subgroup
of it. Suppose $f_1, f_2$ are holomorphic functions on $G$. If
$f_1(x)=f_2(x)$ for all $x\in K$, then $f_1=f_2$ on $G$.
\end{proposition}


\section{$K$-admissible measures}

In this section we introduce the notion of $K$-admissible measures
on complex reductive groups, and prove that $K$-admissible measures
are abundant enough. To do this, we first consider a class of
measures on general Lie groups.

\begin{definition}\label{D:measure}
Let $G$ be a Lie group. A measure $d\mu$ on $G$ is \emph{tame} if it
has the form $d\mu=\mu dg$, where $dg$ is a (left) Haar measure on
$G$, $\mu$ is a measurable function on $G$ which is locally bounded
from below, in the sense that, for every $x\in G$, there is a
$\delta>0$ and a neighborhood $U$ of $x$ such that
$\mu(y)\geq\delta$ for almost all $y\in U$ (with respect to $dg$).
\end{definition}

Let $G$ be a complex Lie group. Denote the space of continuous
functions on $G$ by $C(G)$. For $g\in G$, we define the \emph{right
action $R_g$} and the \emph{left action $L_g$} of $g$ on $C(G)$ by
$$(R_gf)(h)=f(hg),$$
$$(L_gf)(h)=f(g^{-1}h),$$ where $f\in C(G), h\in G$.

\begin{lemma}\label{L:existence}
Let $G$ be a Lie group. Suppose $\mathcal{F}$ is a countable subset
of $C(G)$. Then there exists a tame measure $d\mu$ on $G$ such that
for all $g_1, g_2\in G, f\in\mathcal{F}$, we have
$R_{g_1}L_{g_2}f\in L^{2}(G,d\mu)$.
\end{lemma}

\begin{proof}
Let $\mathcal{F}=\{f_1, f_2, \cdots\}$. Choose a left-invariant
metric $d(\cdot,\cdot)$ on $G$. Then for $g,h\in G$, we have
\begin{equation}\label{E:inequality}
d(e,gh)\leq d(e,g)+d(g,gh)=d(e,g)+d(e,h).
\end{equation}
Let $$K_n=\{g\in G:d(e,g)\leq n\}$$ for each $n\in\mathbb{N}$, and
denote $$M_n=\max\{|f_k(g)|^2:g\in K_n, 1\leq k\leq n\}.$$ For each
$n\in\mathbb{N}$, choose $a_n>0$ such that $$a_n M_{2n}
|K_n\backslash K_{n-1}|\leq\frac{1}{2^n},$$ where $|S|$ denotes the
measure of a subset $S\subset G$ with respect to a fixed Haar
measure $dg$ on $G$. We may also assume that $a_{n+1}\leq a_n$ for
each $n$. Define $\mu(g)=a_n$ when $g\in K_n\backslash K_{n-1}$, and
let $d\mu(g)=\mu(g)dg$. It is obvious that $d\mu$ is a tame measure.
We claim that $d\mu$ satisfies the conclusion of the lemma. In fact,
for any pair $h_1,h_2\in G$, choose a positive integer $N$ such that
$\max\{d(e,h_1), d(e,h_2^{-1})\}\leq \frac{N}{2}$. By the inequality
\eqref{E:inequality}, we have
$$d(e,h_2^{-1}gh_1)\leq d(e,g)+N.$$ This
means that $h_2^{-1}K_nh_1\subset K_{n+N}$ for each
$n\in\mathbb{N}$. Now for each $n\in\mathbb{N}$, we have
\begin{align*}
&\int_G |(R_{h_1}L_{h_2}f_n)(g)|^2\,d\mu(g)\\
=&\int_{K_{N-1}}|f_n(h_2^{-1}gh_1)|^2\mu(g)\,dg+\sum_{n=N}^{+\infty}
\int_{K_{n}\backslash K_{n-1}}|f_n(h_2^{-1}gh_1)|^2\mu(g)\,dg\\
\leq&\int_{K_{N-1}}M_{2N-1}a_1\,dg+\sum_{n=N}^{+\infty}
\int_{K_{n}\backslash K_{n-1}}M_{n+N}a_n\,dg\\
=&M_{2N-1}a_1|K_{N-1}|+\sum_{n=N}^{+\infty}
M_{n+N}a_n|K_{n}\backslash K_{n-1}|\\
\leq&M_{2N-1}a_1|K_{N-1}|+\sum_{n=N}^{+\infty}
M_{2n}a_n|K_{n}\backslash K_{n-1}|\\
\leq&M_{2N-1}a_1|K_{N-1}|+\sum_{n=N}^{+\infty}\frac{1}{2^n}\\
<&+\infty.
\end{align*}
So $R_{h_1}L_{h_2}f_n\in\mathcal{H}L^{2}(G,d\mu)$. This proves the
lemma.
\end{proof}

Now let $G$ be a complex reductive group. Let $\mathcal{H}(G)$
denote the Fr\'{e}chet space of holomorphic functions on $G$, with
respect to the semi-norms $p_X(f)=\sup_{g\in X}|f(g)|$, where $X$ is
any compact subset of $G$. Convergence of a sequence in
$\mathcal{H}(G)$ with respect to the Fr\'echet topology is
equivalent to locally uniform convergence. For a tame measure $d\mu$
on $G$, let $\mathcal{H}L^{2}(G,d\mu)$ denote the space of
$L^2$-$d\mu$-integrable functions in $\mathcal{H}(G)$, with the
inner product
$$
\langle f_{1},f_{2}\rangle=\int_{G}f_{1}(g)\overline{f_{2}(g)}\,
d\mu(g).
$$
If a sequence $f_1, f_2, \cdots$ in $\mathcal{H}L^{2}(G,d\mu)$
$L^2$-converges to a measurable function $f$, then it also converges
locally uniformly to $f$. This implies that
$f\in\mathcal{H}L^{2}(G,d\mu)$. Hence $\mathcal{H}L^{2}(G,d\mu)$ is
a Hilbert space.

Let $K$ be a maximal compact subgroup of $G$. For a
finite-dimensional holomorphic representation $\pi$ of $G$ with
representation space $V_{\pi}$ of dimension $d_\pi$, we always
choose a fixed inner product $\langle\cdot,\cdot\rangle$ on
$V_{\pi}$ such that the restriction of $\pi$ to $K$ is unitary. We
also choose a fixed orthonormal basis $\{e_{j}\}$ of $V_{\pi}$ with
respect to this inner product, and let
$$
\pi_{ij}(g)=\langle\pi(g)e_{j},e_{i}\rangle
$$
be the matrix elements. For an $n\times n$ matrix $A$, we let
$\|A\|^{2}=\mathrm{tr}(AA^{*})=\sum_{i,j=1}^{n}|A_{ij}|^{2}$. Then
we define
$$
C_{\pi,\mu}=\int_{G}\|\pi(g)\|^{2}\, d\mu(g)
$$
for a tame measure $d\mu$ on $G$.

\begin{definition}\label{D:measure}
Let the notations be as above. A tame measure $d\mu$ on $G$ is
\emph{$K$-admissible} if it is $K$-bi-invariant and such that
$C_{\pi,\mu}<+\infty$ for each $[\pi]\in\widehat{G}_{\mathcal{H}}$.
\end{definition}

\begin{theorem}\label{T:existence}
Let $G$ be a complex reductive group, $\mathcal{F}$ be a countable
subset of $\mathcal{H}(G)$. Then for any maximal compact subgroup
$K$ of $G$, there exists a $K$-admissible measure $d\mu$ such that
for all $g_1,g_2\in G, f\in\mathcal{F}$, we have
$R_{g_1}L_{g_2}f\in\mathcal{H}L^{2}(G,d\mu)$. In particular, we have
$\mathcal{F}\subset\mathcal{H}L^{2}(G,d\mu)$.
\end{theorem}

\begin{proof}
Applying Lemma \ref{L:existence} to the countable set of holomorphic
functions
$\mathcal{F}'=\mathcal{F}\bigcup\{\pi_{ij}:i,j=1,\cdots,d_{\pi},
[\pi]\in\widehat{G}_{\mathcal{H}}\}$, we get a tame measure $d\nu$
on $G$ such that $R_{g_1}L_{g_2}\pi_{ij},
R_{g_1}L_{g_2}f\in\mathcal{H}L^{2}(G,d\nu)$ for all $g_1,g_2\in G,
f\in\mathcal{F}$. Suppose $d\nu(g)=\nu(g)dg$, and let
$$\mu(g)=\int_K\int_K\nu(xgy)\,dxdy,$$ where $dx, dy$ refer to the
Haar measure on $K$. Then $\mu(g)$ is a $K$-bi-invariant function on
$G$ locally bounded from below, and then the measure
$d\mu(g)=\mu(g)dg$ is a $K$-bi-invariant tame measure. We claim that
if a function $f\in\mathcal{H}(G)$ such that
$R_{x_1}L_{x_2}f\in\mathcal{H}L^{2}(G,d\nu)$ for all $x_1,x_2\in K$,
then $f\in\mathcal{H}L^{2}(G,d\mu)$. In fact,
\begin{align*}
&\int_G |f(g)|^2\,d\mu(g)\\
=&\int_G |f(g)|^2\mu(g)\,dg\\
=&\int_G\int_K\int_K |f(g)|^2\nu(xgy)\,dxdydg\\
=&\int_K\int_K \left(\int_G |f(x^{-1}gy^{-1})|^2\nu(g)\,dg\right)\,dxdy\\
=&\int_K\int_K \left(\int_G
|(R_{y^{-1}}L_{x}f)(g)|^2\,d\nu(g)\right)\,dxdy.
\end{align*}
By the assumption, $\int_G
|(R_{y^{-1}}L_{x}f)(g)|^2\,d\nu(g)<+\infty$, and a standard
analytical argument shows that it is continuous in $(x,y)$. So
$\int_G |f(g)|^2\,d\mu(g)<+\infty$, that is,
$f\in\mathcal{H}L^{2}(G,d\mu)$. Applying this fact to the functions
$\pi_{ij}, R_{g_1}L_{g_2}f \ (f\in\mathcal{F})$, we get
$\pi_{ij}\in\mathcal{H}L^{2}(G,d\mu)$ (which means that
$C_{\pi,\mu}<+\infty$) and
$R_{g_1}L_{g_2}f\in\mathcal{H}L^{2}(G,d\mu)$. So the measure $d\mu$
is $K$-admissible and satisfies the conclusion of the theorem.
\end{proof}


\section{Theorems of Peter-Weyl-type}

The classical Peter-Weyl Theorems claim that the linear span of
matrix elements of irreducible representations of a compact group is
uniformly dense in the Banach space continuous functions on the
group, and is $L^2$-dense in the Hilbert space of $L^2$-integrable
functions, which are the starting point of harmonic analysis on
compact groups. We prove in this section the similar results for
complex reductive groups, which are the base of holomorphic harmonic
analysis on such groups. The $L^2$ analog is due to Hall \cite{Ha},
Theorems 9 and 10. But our proof of the completeness part of the
$L^2$ Peter-Weyl-type theorem is shorter, which makes use of
representation theory of compact groups.

Roughly speaking, the $L^2$ Peter-Weyl-type theorem claim that
certain regular representation of the complex reductive group is
completely reducible. We first give the precise definition.

\begin{definition}\label{D:regular}
Let $G$ be a complex Lie group with a tame measure $d\mu$. The
\emph{right} and \emph{left $L^2$ holomorphic regular
representations} $\pi_{\mu,R}$ and $\pi_{\mu,L}$ of $G$ on
$\mathcal{H}L^{2}(G,d\mu)$ are defined by
$$
(\pi_{\mu,R}(g)f)(h)=f(hg)
$$
and
$$
(\pi_{\mu,L}(g)f)(h)=f(g^{-1}h)
$$
respectively, where $f\in\mathcal{H}L^{2}(G,d\mu), \; g,h\in G$. For
an element $g\in G$, the domains of $\pi_{\mu,R}(g)$ and
$\pi_{\mu,L}(g)$ are
$$
\mathcal{D}(\pi_{\mu,R}(g))=\{f\in\mathcal{H}L^{2}(G,d\mu):\int_{G}|f(hg)|^{2}\,d\mu(h)<+\infty\}
$$
and
$$
\mathcal{D}(\pi_{\mu,L}(g))=\{f\in\mathcal{H}L^{2}(G,d\mu):
\int_{G}|f(g^{-1}h)|^{2}\,d\mu(h)<+\infty\}
$$
respectively.
\end{definition}

\begin{remark}
One should notice here that $\pi_{\mu,R}(g)$ and $\pi_{\mu,L}(g)$
are unbounded operators in general, that is, their domains are not
necessarily the whole space $\mathcal{H}L^{2}(G,d\mu)$. But if some
$g\in G$ such that $(r_g)_*d\mu=d\mu$ (or $(l_g)_*d\mu=d\mu$), then
$\mathcal{D}(\pi_{\mu,R}(g))$ (or $\mathcal{D}(\pi_{\mu,L}(g))$) is
the whole $\mathcal{H}L^{2}(G,d\mu)$. If $G$ is complex reductive
and $d\mu$ is $K$-admissible, we will show that there is a dense
subspace $\mathcal{E}$ of $\mathcal{H}L^{2}(G,d\mu)$ such that
$\mathcal{E}\subset\mathcal{D}(\pi_{\mu,R}(g))\cap
\mathcal{D}(\pi_{\mu,L}(g))$ for all $g\in G$.
\end{remark}

Let $G$ be a complex reductive group, $K$ a maximal compact subgroup
of it. For a finite-dimensional holomorphic representation $\pi$ of
$G$ with representation space $V_{\pi}$ of dimension $d_\pi$, we
denote the linear span of the matrix elements
$\{\pi_{ij}:i,j=1,\cdots,d_{\pi}\}$ of $\pi$ by $\mathcal{E}_{\pi}$,
and let $\mathcal{E}$ be the linear span of
$\{\mathcal{E}_{\pi}:[\pi]\in\widehat{G}_{\mathcal{H}}\}$. Note that
for a tame measure $\mu$ on $G$,
$C_{\pi,\mu}=\int_{G}\|\pi(g)\|^{2}\, d\mu(g)<+\infty$ if and only
if $\mathcal{E}_{\pi}\subset\mathcal{H}L^{2}(G,d\mu)$. So if $d\mu$
is $K$-admissible, then
$\mathcal{E}\subset\mathcal{H}L^{2}(G,d\mu)$. We let
$\widetilde{\pi}(g)=\pi(g^{-1})^{t}$, where $A^{t}$ denotes the
transpose of a matrix $A$. Then $\widetilde{\pi}$ is also a
holomorphic representation of $G$, and $\widetilde{\pi}$ is
irreducible if and only if $\pi$ is irreducible.

\begin{theorem}\label{T:Peter-Weyl}
Let $G$ be a complex reductive group with a maximal compact subgroup
$K$ and a $K$-admissible measure $d\mu$. Then we have
{\flushleft(i)}
$\mathcal{H}L^{2}(G,d\mu)=\bigoplus_{[\pi]\in\widehat{G}_{\mathcal{H}}}\mathcal{E}_{\pi}$,
and the set
$$\{\frac{d_{\pi}}{\sqrt{C_{\pi,\mu}}}\pi_{ij}:i,j=1,\cdots,d_{\pi},
[\pi]\in\widehat{G}_{\mathcal{H}} \}$$ is an orthonormal basis of
$\mathcal{H}L^{2}(G,d\mu)$. {\flushleft(ii)} For each $g\in G$, the
domains $\mathcal{D}(\pi_{\mu,R}(g))$ and
$\mathcal{D}(\pi_{\mu,L}(g))$ contain the linear span $\mathcal{E}$
of $\{\mathcal{E}_{\pi}:[\pi]\in\widehat{G}_{\mathcal{H}}\}$, and
hence are dense in $\mathcal{H}L^{2}(G,d\mu)$. {\flushleft(iii)}
Both $\pi_{\mu,R}(g)$ and $\pi_{\mu,L}(g)$ are completely reducible
and can be splited as direct sums of elements of
$\widehat{G}_{\mathcal{H}}$, each
$[\pi]\in\widehat{G}_{\mathcal{H}}$ occurs with multiplicity
$d_{\pi}$. {\flushleft(iv)} For $1\leq i\leq d_\pi$, the subspace of
$\mathcal{E}_{\pi}$ (resp. $\mathcal{E}_{\widetilde{\pi}}$) spanned
by the $i$-th row (resp. the $i$-th column) of the matrix
$(\pi_{ij})$ (resp. $(\widetilde{\pi}_{ij})$) is invariant under
$\pi_{\mu,R}$ (resp. $\pi_{\mu,L}$), and the restriction of
$\pi_{\mu,R}$ (resp. $\pi_{\mu,L}$) to this subspace is equivalent
to $\pi$.
\end{theorem}

\begin{proof}
We first show that for any
$[\pi],[\pi']\in\widehat{G}_{\mathcal{H}}$, we have
\begin{equation}\label{e:formula}
\int_{G}\pi_{ij}(g)\overline{\pi'_{kl}(g)}\,d\mu(g)=
\begin{cases}
\frac{C_{\pi,\mu}}{d_\pi^2}, &\text{if}\; [\pi]=[\pi'], i=k, j=l;\\
0, &\text{otherwise}.
\end{cases}
\end{equation}

For each linear transform $A:V_{\pi'}\to V_{\pi}$, define
$$
\overline{A}=\int_{G}\pi(g)A\pi'(g)^{*}\,d\mu(g),
$$
where $\pi'(g)^{*}$ is the adjoint of $\pi'(g)$ with respect to the
chosen inner product on $V_{\pi'}$. For each $x\in K$, by the
$K$-left-invariance of $d\mu$, we can easily prove that
$\pi(x)\overline{A}\pi'(x)^{-1}=\overline{A}$. So $\overline{A}$ is
an intertwining operator between $\pi|_{K}$ and $\pi'|_{K}$. Since
$\pi|_{K}$ and $\pi'|_{K}$ are irreducible, by Schur's Lemma, we
have
$$
\overline{A}=\int_{G}\pi(g)A\pi'(g)^{*}\,d\mu(g)=
\begin{cases}
0, &[\pi]\neq[\pi'];\\
c_{A}I_{d_{\pi}}, &\pi=\pi',
\end{cases}
$$
where $c_{A}$ is a constant. If $[\pi]\neq[\pi']$, by the
arbitrariness of $A$, we have
$$
\int_{G}\pi_{ij}(g)\overline{\pi'_{kl}(g)}\,d\mu(g)=0.
$$
That is, $\mathcal{E}_{\pi}\bot\mathcal{E}_{\pi'}$. Now assume
$\pi=\pi'$. Then
$$
\overline{A}_{ij}=\sum_{k,l}A_{kl}\int_{G}\pi_{ik}(g)\overline{\pi_{jl}(g)}\,d\mu(g)
=c_{A}\delta_{ij},
$$
where $(A_{ij})$ is the matrix form of $A$ with respect to the
chosen bases of $V_{\pi}$. If $i\neq j$, we have
$$
\int_{G}\pi_{ik}(g)\overline{\pi_{jl}(g)}\,d\mu(g)=0, \qquad i\neq
j.
$$
Hence two matrix elements of $\pi$ which lie in different row are
orthogonal. Now take $i=j$, $(A_{kl})=E_{kk}$, where $E_{kk}$ is the
matrix with $1$ at the $(k,k)$-position and $0$ elsewhere, then we
obtain
$$
\int_{G}\pi_{ik}(g)\overline{\pi_{ik}(g)}\,d\mu(g)=c_{E_{kk}}
$$
for $1\leq i\leq d_\pi$. Hence the matrix elements lying in the same
column have the same norm.

For an endomorphism $A$ of $V_{\pi}$, we define
$$
\overline{\overline{A}}=\int_{G}\pi(g)^{*}A\pi(g)\,d\mu(g).
$$
Then, similarly, by the $K$-right-invariance of $d\mu$,
$\overline{\overline{A}}$ is an intertwining operator of $\pi|_{K}$
and then $ \overline{\overline{A}}=c'_{A}I_{d_{\pi}} $ for some
constant $c'_{A}$. So we have
$$
\overline{\overline{A}}_{ij}
=\sum_{k,l}A_{kl}\int_{G}\pi_{lj}(g)\overline{\pi_{ki}(g)}\,d\mu(g)=c'_{A}\delta_{ij}.
$$
Then
$$
\int_{G}\pi_{lj}(g)\overline{\pi_{ki}(g)}\,d\mu(g)=0, \qquad i\neq
j,
$$
$$
\int_{G}\pi_{ki}(g)\overline{\pi_{ki}(g)}\,d\mu(g)=c'_{E_{kk}}.
$$
Hence two matrix elements which lie in different columns are
orthogonal and the matrix elements lying in the same row  have the
same norm.

Combining the above results, we get the conclusion that two
different matrix elements of $\pi$ are orthogonal and all matrix
elements of $\pi$ have the same norm. Moreover, we have
\begin{align*}
&\int_{G}\pi_{ij}(g)\overline{\pi_{ij}(g)}\,d\mu(g)\\
=&\frac{1}{d_{\pi}^{2}}\int_{G}\sum_{k,l=1}^{d_\pi}\pi_{kl}(g)\overline{\pi_{kl}(g)}\,d\mu(g)\\
=&\frac{1}{d_{\pi}^{2}}\int_{G}\|\pi(g)\|^{2}\,d\mu(g)\\
=&\frac{C_{\pi,\mu}}{d_{\pi}^{2}}.
\end{align*}
This completes the proof of \eqref{e:formula}.

Now for each $[\pi]\in\widehat{G}_{\mathcal{H}}$, consider the
action of $\pi_{\mu,R}$ and $\pi_{\mu,L}$ on $\mathcal{E}_{\pi}$. We
have
\begin{align*}
(\pi_{\mu,R}(g)\pi_{ij})(h)=&\pi_{ij}(hg)=
\sum_{k=1}^{d_{\pi}}\pi_{ik}(h)\pi_{kj}(g),\\
(\pi_{\mu,L}(g)\pi_{ij})(h)=&\pi_{ij}(g^{-1}h)=
\sum_{k=1}^{d_{\pi}}\widetilde{\pi}_{ki}(g)\pi_{kj}(h).
\end{align*}
That is,
\begin{align*}
\pi_{\mu,R}(g)\pi_{ij}=&\sum_{k=1}^{d_{\pi}}\pi_{kj}(g)\pi_{ik},\\
\pi_{\mu,L}(g)\pi_{ij}=&\sum_{k=1}^{d_{\pi}}\widetilde{\pi}_{ki}(g)\pi_{kj}.
\end{align*}
Hence $\mathcal{E}_{\pi}$ is contained in the domains of
$\pi_{\mu,R}(g)$ and $\pi_{\mu,L}(g)$, and for $1\leq i\leq d_\pi$,
the subspace of $\mathcal{E}_{\pi}$ spanned by the $i$-th row (resp.
the $i$-th column) of the matrix $(\pi_{ij})$ is invariant under
$\pi_{\mu,R}$ (resp. $\pi_{\mu,L}$), and the restriction of
$\pi_{\mu,R}$ (resp. $\pi_{\mu,L}$) to this subspace is equivalent
to $\pi$ (resp. $\widetilde{\pi}$ ). For $\pi_{\mu,L}$, its
restriction to the subspace of $\mathcal{E}_{\widetilde{\pi}}$
spanned by the $i$-th column of $(\widetilde{\pi}_{ij})$ is
equivalent to $\widetilde{\widetilde{\pi}}=\pi$.

The set (4.2) spans a closed subspace
$V=\overline{\mathcal{E}}=\bigoplus_{[\pi]
\in\widehat{G}_{\mathcal{H}}}\mathcal{E}_{\pi}$ of
$\mathcal{H}L^{2}(G,d\mu)$. We prove that
$V=\mathcal{H}L^{2}(G,d\mu)$. If not, since $V$ is invariant under
$\pi_{\mu,R}$ and then invariant under $\pi_{\mu,R}|_{K}$, which is
a unitary representation of $K$, $V^{\bot}\neq0$ is also invariant
under $\pi_{\mu,R}|_{K}$. Because a representation of a compact
group is completely reducible (Folland \cite{Fo}, Theorem 5.2), we
can choose an irreducible subspace $W$ of $V^{\bot}$ under
$\pi_{\mu,R}|_{K}$. By Proposition \ref{T:representation}, the
restriction of $\pi_{\mu,R}|_{K}$ on $W$ is equivalent to $\pi|_{K}$
for some $[\pi]\in\widehat{G}_{\mathcal{H}}$. We choose a suitable
basis $\{f_{1},\cdots,f_{d_{\pi}}\}$ of $W$ such that with respect
to this basis, the matrix of the subrepresentation of
$\pi_{\mu,R}|_{K}$ on $W$ is $(\pi_{ij}|_{K})$. Then
$$
f_{i}(gx)=(\pi_{\mu,R}(x)f_{i})(g)=\sum_{j=1}^{{d_{\pi}}}\pi_{ji}(x)f_{j}(g),
\quad x\in K, g\in G, i=1,\cdots,d_{\pi}.
$$
Let $g=e$, we have
$$
f_{i}(x)=\sum_{j=1}^{{d_{\pi}}}f_{j}(e)\pi_{ji}(x), \qquad x\in K.
$$
Since $f_{i}$'s and $\pi_{ij}$'s are all holomorphic on $G$, by
Proposition \ref{T:extension}, we have
$$
f_{i}=\sum_{j=1}^{{d_{\pi}}}f_{j}(e)\pi_{ji}.
$$
So $f_{i}\in \mathcal{E}_{\pi}$, which contradicts to $W\bot V$. The
proof of the theorem is completed.
\end{proof}

\begin{theorem}
Let $G$ be a complex reductive group. Then $\mathcal{E}$ is a dense
subspace of the Fr\'echet space $\mathcal{H}(G)$.
\end{theorem}

\begin{proof}
Choose a maximal compact subgroup $K$ of $G$. Let
$f\in\mathcal{H}(G)$. By Theorem \ref{T:existence}, there is a
$K$-admissible measure $d\mu$ on $G$ such that
$f\in\mathcal{H}L^2(G,d\mu)$. By Theorem \ref{T:Peter-Weyl}, $f$
lies in the closure of $\mathcal{E}$ with respect to the Hilbert
topology on $\mathcal{H}L^2(G,d\mu)$, hence lies in the Fr\'echet
topology on $\mathcal{H}(G)$. This proves the theorem.
\end{proof}


\section{Holomorphic Fourier transforms on complex reductive groups}

As in the previous section, let $G$ be a complex reductive group.
Choose a maximal compact subgroup $K$ of $G$. For each holomorphic
function $f\in\mathcal{H}(G)$, by Theorem \ref{T:existence}, there
exists a $K$-admissible measure $d\mu$ such that
$f\in\mathcal{H}L^{2}(G,d\mu)$. Then by Theorem \ref{T:Peter-Weyl},
we can expand $f$ as
\begin{equation}\label{e:6.1}
f=\sum_{[\pi]\in\widehat{G}_{\mathcal{H}}}\sum_{i,j=1}^{d_\pi}\lambda_{\pi,ij}\pi_{ij}\;,
\end{equation}
where
\begin{equation}
\lambda_{\pi,ij}=\frac{d_\pi^2}{C_{\pi,\mu}}\int_{G}f(g)\overline{\pi_{ij}(g)}\,d\mu(g).
\end{equation}
Equation \eqref{e:6.1} is called the \emph{holomorphic Fourier
series} of $f$. It obviously converges in $\mathcal{H}L^{2}(G,d\mu)$
by Theorem \ref{T:Peter-Weyl}, hence also converges locally
uniformly.

\begin{remark}
When $G=\mathbb{C}^{*}$, the Fourier series \eqref{e:6.1} is just
the Laurant series in complex analysis.
\end{remark}

Now we give the definition of the holomorphic Fourier transform of
$f\in\mathcal{H}(G)$ as follows.

\begin{definition}
For $f\in\mathcal{H}(G)$, its \emph{holomorphic Fourier transform}
$\widehat{f}$ is defined by
\begin{equation}
\widehat{f}(\pi)=\frac{1}{C_{\pi,\mu}}\int_{G}f(g)\pi(g)^*\,d\mu(g),
\qquad [\pi]\in\widehat{G}_{\mathcal{H}},
\end{equation}
where $d\mu$ is a $K$-admissible measure on $G$ such that
$f\in\mathcal{H}L^{2}(G,d\mu)$.
\end{definition}

Note that $\widehat{f}(\pi)$ is an operator in the representation
space $V_\pi$ of $\pi$. With the orthonormal basis $\{e_{j}\}$ of
$V_{\pi}$ that we have chosen, $\widehat{f}(\pi)$ is given in the
matrix form
\begin{equation}
\widehat{f}(\pi)_{ij}=\frac{1}{C_{\pi,\mu}}\int_{G}f(g)
\overline{\pi_{ji}(g)}\,d\mu(g).
\end{equation}
Comparing this equation with formula (5.2), we have
\begin{equation}
\widehat{f}(\pi)_{ij}=\frac{1}{d_\pi^2}\lambda_{\pi,ji}.
\end{equation}

The following proposition shows that the Fourier transform is
independent of the choice of $K$-admissible measures. To avoid the
ambiguity, we denote the holomorphic Fourier transform of
$f\in\mathcal{H}(G)$ by $\widehat{f}_\mu$ for the moment, when the
$K$-admissible measure involved in the definition is $d\mu$.

\begin{proposition}
Let $d\mu$ and $d\nu$ be two $K$-admissible measures on $G$. If
$f\in\mathcal{H}L^{2}(G,d\mu)\bigcap\mathcal{H}L^{2}(G,d\nu)$, then
$\widehat{f}_\mu=\widehat{f}_\nu$.
\end{proposition}

\begin{proof}
To be precise, we rewrite the symbol $\lambda_{\pi,ij}$ defined in
(5.2) by $\lambda^\mu_{\pi,ij}$. So by \eqref{e:6.1}, we have
$$f=\sum_{[\pi]\in\widehat{G}_{\mathcal{H}}}\sum_{i,j=1}^{d_\pi}\lambda^\mu_{\pi,ij}\pi_{ij}\;.$$
Let $K$ be a maximal compact subgroup with the normalized Haar
measure $dx$. Since the above series converges locally uniformly, it
converges uniformly on $K$. But on the compact group $K$, uniform
convergence implies convergence in $L^2(K,dx)$, so the series
$$f|_K=\sum_{[\pi]\in\widehat{G}_{\mathcal{H}}}\sum_{i,j=1}^{d_\pi}\lambda^\mu_{\pi,ij}\pi_{ij}|_K$$
is in fact the Fourier series of $f|_K$ on $K$. Hence the
coefficient $\lambda^\mu_{\pi,ij}$ is uniquely determined by $f|_K$,
which is independent of $\mu$. That is,
$\lambda^\mu_{\pi,ij}=\lambda^\nu_{\pi,ij}$. By equation (5.5), we
get $\widehat{f}_\mu=\widehat{f}_\nu$.
\end{proof}

Thanks to this proposition, we can omit the subscription and simply
denote the Fourier transform of $f$ by $\widehat{f}$.

\begin{theorem}
For $f\in\mathcal{H}(G)$, we have the Fourier inversion formula
\begin{equation}
f(g)=\sum_{[\pi]\in\widehat{G}_{\mathcal{H}}}
d_\pi^2\,tr(\widehat{f}(\pi)\pi(g)).
\end{equation}
The series converges locally uniformly. If $d\mu$ is an
$K$-admissible measure such that $f\in\mathcal{H}L^{2}(G,d\mu)$,
then the series also converges in $\mathcal{H}L^{2}(G,d\mu)$.
\end{theorem}

\begin{proof}
Suppose $d\mu$ is a $K$-admissible measure with
$f\in\mathcal{H}L^{2}(G,d\mu)$. By equation (5.5), we have
\begin{align*}
\sum_{i,j=1}^{d_\pi}\lambda_{\pi,ij}\pi_{ij}(g)
=&d_\pi^2\sum_{i,j=1}^{d_\pi}\widehat{f}(\pi)_{ji}\pi_{ij}(g)\\
=&d_\pi^2\,tr(\widehat{f}(\pi)\pi(g)).
\end{align*}
By \eqref{e:6.1}, we get
$$f(g)=\sum_{[\pi]\in\widehat{G}_{\mathcal{H}}}
d_\pi^2\,tr(\widehat{f}(\pi)\pi(g)).$$ It obviously converges in
$\mathcal{H}L^{2}(G,d\mu)$, hence also converges locally uniformly.
\end{proof}

\begin{corollary}
For $f\in\mathcal{H}(G)$, we have
\begin{equation}
f(e)=\sum_{[\pi]\in\widehat{G}_{\mathcal{H}}}
d_\pi^2\,tr(\widehat{f}(\pi)).
\end{equation}
\end{corollary}\qed

Equation (5.7) can be viewed as a version of Plancherel Theorem,
which reflects the completeness of the holomorphic Fourier analysis.

Now we give the relation between the Fourier transform of a function
$f\in\mathcal{H}(G)$ and the Fourier transform of $R_gf$ and $L_gf$
for $g\in G$, where $R_g$ and $L_g$ are the right and left actions
of $g$ on $\mathcal{H}(G)$, respectively.

\begin{theorem}
For $g\in G$ and $f\in\mathcal{H}(G)$, we have
\begin{equation}
(R_gf)\,\widehat{}\,(\pi)=\pi(g)\widehat{f}(\pi).
\end{equation}
\begin{equation}
(L_gf)\,\widehat{}\,(\pi)=\widehat{f}(\pi)\pi(g^{-1}).
\end{equation}
\end{theorem}

\begin{proof}
We first prove (5.8). By Theorem \ref{T:existence}, we can choose a
$K$-admissible measure $d\mu$ on $G$ such that
$R_gf\in\mathcal{H}L^{2}(G,d\mu)$ for all $g\in G$. For $x\in K$,
$\pi(x)$ is unitary, so we have
\begin{align*}
(R_xf)\,\widehat{}\,(\pi)
=&\frac{1}{C_{\pi,\mu}}\int_{G}f(hx)\pi(h)^*\,d\mu(h)\\
=&\frac{1}{C_{\pi,\mu}}\int_{G}f(h)\pi(hx^{-1})^*\,d\mu(h)\\
=&\frac{1}{C_{\pi,\mu}}\int_{G}f(h)\pi(x^{-1})^*\pi(h)^*\,d\mu(h)\\
=&\frac{\pi(x)}{C_{\pi,\mu}}\int_{G}f(h)\pi(h)^*\,d\mu(h)\\
=&\pi(x)\widehat{f}(\pi).
\end{align*}
That is, $(R_xf)\,\widehat{}\,(\pi)=\pi(x)\widehat{f}(\pi)$ for all
$x\in K$. But $(R_gf)\,\widehat{}\,(\pi)$ and
$\pi(g)\widehat{f}(\pi)$ are holomorphic with respect to $g\in G$,
by Proposition \ref{T:extension},
$(R_gf)\,\widehat{}\,(\pi)=\pi(g)\widehat{f}(\pi)$ for all $g\in G$.
This proves (5.8). The proof of (5.9) is similar.
\end{proof}

We now consider the Fourier transform of a function under the action
of invariant differential operators. Recall that a left-invariant
differential operator $D$ on $G$ can be viewed as an element of the
universal enveloping algebra $U(\mathfrak{g})$ of the Lie algebra
$\mathfrak{g}$ of $G$, and a finite-dimensional representation $\pi$
of $G$ induces a representation of $U(\mathfrak{g})$, which we also
denote by $\pi$.

\begin{corollary}\label{C:D}
Suppose $D$ is a left-invariant differential operator on $G$, then
for $f\in\mathcal{H}(G)$, we have
\begin{equation}
(Df)\,\widehat{}\,(\pi)=\pi(D)\widehat{f}(\pi).
\end{equation}
\end{corollary}

\begin{proof}
By the Poincar\'{e}-Birkhoff-Witt Theorem, $D$ can be expressed as a
linear combination of the monomials $X_1X_2\cdots X_r$, where
$X_i\in\mathfrak{g}$. So without loss of generality, we may assume
$D=X\in\mathfrak{g}$. Now
\begin{align*}
(Xf)(g)=&\frac{d}{dt}\Big|_{t=0}f(ge^{tX})\\
=&\frac{d}{dt}\Big|_{t=0}(R_{e^{tX}}f)(g).
\end{align*}
So by (5.8),
\begin{align*}
(Xf)\,\widehat{}\,(\pi)
=&\frac{d}{dt}\Big|_{t=0}(R_{e^{tX}}f)\,\widehat{}\,(\pi)\\
=&\frac{d}{dt}\Big|_{t=0}\pi(e^{tX})\widehat{f}(\pi)\\
=&\pi(X)\widehat{f}(\pi).\\
\end{align*}
This completes the proof.
\end{proof}

For right-invariant differential operator we have a similar
argument, which is deduced from (5.9). We state the conclusion
below. Its proof is similar to that of the above corollary, we omit
it here.

\begin{corollary}
Suppose $D'$ is a right-invariant differential operator on $G$, then
for $f\in\mathcal{H}(G)$, we have
\begin{equation}
(D'f)\,\widehat{}\,(\pi)=\widehat{f}(\pi)\pi(D').
\end{equation}
\end{corollary}\qed


\section{Applications of the Holomorphic Fourier transform\label{7}}

In this section, we give two applications of the holomorphic Fourier
transform. One is to holomorphic class functions on complex
reductive groups, the other is to holomorphic evolution PDEs on such
groups.

First we prove that the linear span $\mathcal{E}_c$ of irreducible
holomorphic characters of a complex reductive group $G$ is dense in
the Fr\'echet space $\mathcal{H}_c(G)$ of holomorphic class
functions on $G$. We prove this by showing that the irreducible
holomorphic characters of a $G$ form an orthogonal basis of the
subspace of class functions in $\mathcal{H}L^{2}(G,d\mu)$, where
$d\mu$ is a $K$-admissible measure on $G$. For a finite-dimensional
holomorphic representation $\pi$ of $G$, the character $\chi_\pi$ of
$\pi$ is defined by $\chi_\pi(g)=\mathrm{tr}\,\pi(g)$. A function
$f\in\mathcal{H}(G)$ is called a class function if $f(gh)=f(hg)$ for
all $g,h\in G$. It is obvious that holomorphic characters are class
functions. We denote the linear span of irreducible holomorphic
characters by $\mathcal{E}_c$, the Fr\'echet space of class
holomorphic functions by $\mathcal{H}_c(G)$, and the subspace of
class functions in $\mathcal{H}L^{2}(G,d\mu)$ by
$\mathcal{H}L_c^{2}(G,d\mu)$.

\begin{theorem}\label{T:class}
(i)
$\{\sqrt{\frac{d_{\pi}}{C_{\pi,\mu}}}\;\chi_\pi:[\pi]\in\widehat{G}_{\mathcal{H}}
\}$ is an orthonormal basis of $\mathcal{H}L_c^{2}(G,d\mu)$.\\
(ii) $\mathcal{E}_c$ is dense in $\mathcal{H}_c(G)$ with respect to
the Fr\'echet topology.
\end{theorem}

\begin{proof}
By Theorem \ref{T:Peter-Weyl}, it is obvious that
$\{\sqrt{\frac{d_{\pi}}{C_{\pi,\mu}}}\;\chi_\pi:[\pi]\in\widehat{G}_{\mathcal{H}}
\}$ is an orthonormal set in $\mathcal{H}L_c^{2}(G,d\mu)$. To prove
the completeness, it is enough to show that each
$f\in\mathcal{H}L_c^{2}(G,d\mu)$ can be expressed as a series of the
form $\sum_{[\pi]\in\widehat{G}_{\mathcal{H}}}a_\pi\chi_\pi$. Notice
that for each $[\pi]\in\widehat{G}_{\mathcal{H}}$, $\pi|_K$ is
unitary, So for $x\in K$, we have
\begin{align*}
\pi(x)\widehat{f}(\pi)\pi(x)^{-1}
=&\frac{1}{C_{\pi,\mu}}\int_{G}f(g)\pi(x^{-1})^*\pi(g)^*\pi(x)^*\,d\mu(g)\\
=&\frac{1}{C_{\pi,\mu}}\int_{G}f(g)\pi(xgx^{-1})^*\,d\mu(g)\\
=&\frac{1}{C_{\pi,\mu}}\int_{G}f(x^{-1}gx)\pi(g)^*\,d\mu(g)\\
=&\frac{1}{C_{\pi,\mu}}\int_{G}f(g)\pi(g)^*\,d\mu(g)\\
=&\widehat{f}(\pi).
\end{align*}
This shows $\pi(g)\widehat{f}(\pi)$ and $\widehat{f}(\pi)\pi(g)$
agree on $K$. But they are holomorphic on $G$, by Proposition
\ref{T:extension}, $\pi(g)\widehat{f}(\pi)=\widehat{f}(\pi)\pi(g)$
for all $g\in G$. Since $\pi$ is irreducible, by Schur's Lemma,
$\widehat{f}(\pi)=b_\pi I$ for some constant $b_\pi$. By the Fourier
inversion formula (5.6), we have
$f=\sum_{[\pi]\in\widehat{G}_{\mathcal{H}}}d_\pi^2b_\pi\chi_\pi$,
which converges both in $\mathcal{H}L_c^{2}(G,d\mu)$ and locally
uniformly. This proves the theorem.
\end{proof}

\begin{corollary}
Every $f\in\mathcal{H}_c(G)$ can be expanded uniquely as
$$f=\sum_{[\pi]\in\widehat{G}_{\mathcal{H}}}a_\pi\chi_\pi.$$
The series converges locally uniformly. The coefficients $a_\pi$ can
be expressed as
$$a_\pi=\frac{d_{\pi}}{C_{\pi,\mu}}\int_{G}f(g)\overline{\chi_\pi(g)}\,d\mu(g)$$
for each $K$-admissible measure $d\mu$ on $G$ such that
$f\in\mathcal{H}L^{2}(G,d\mu).$
\end{corollary}

\begin{proof}
Choose a $K$-admissible measure $d\mu$ such that
$f\in\mathcal{H}L^{2}(G,d\mu)$ and apply Theorem \ref{T:class}.
\end{proof}

Our next application of the holomorphic Fourier transform concerns
holomorphic evolution partial differential equations on a complex
reductive group $G$. Suppose $D$ is a left-invariant differential
operator on $G$. Consider the evolution equation with holomorphic
initial condition
\begin{equation}\label{E:PDE}
\begin{cases}
\frac{d}{dt}f_t=Df_t\\
f_0\in\mathcal{H}(G)
\end{cases}
\end{equation}
on $G$. We have the following result.

\begin{theorem}
If equation \eqref{E:PDE} has a holomorphic solution
$f_t\in\mathcal{H}(G)$ for $t>0$, then the solution $f_t$ can be
expressed as
\begin{equation}\label{E:solution}
f_t(g)=\sum_{[\pi]\in\widehat{G}_{\mathcal{H}}}
d_\pi^2\,tr(e^{t\pi(D)}\widehat{f_0}(\pi)\pi(g)).
\end{equation}
The series converges locally uniformly. Conversely, if the series in
the right hand side of \eqref{E:solution} converges locally
uniformly, then it converges to a holomorphic solution of equation
\eqref{E:PDE}.
\end{theorem}

\begin{proof}
Suppose equation \eqref{E:PDE} has a holomorphic solution
$f_t\in\mathcal{H}(G)$ for $t>0$. Applying the holomorphic Fourier
transform to both side of the equation, by Corollary \ref{C:D}, we
get
$$\frac{d}{dt}\widehat{f_t}(\pi)=(Df_t)\,\widehat{}\,(\pi)=\pi(D)\widehat{f_t}(\pi).$$
This ordinary differential equation of matrix form can be solved
directly as
$$\widehat{f_t}(\pi)=e^{t\pi(D)}\widehat{f_0}(\pi).$$
By the Fourier inversion formula (5.6), we get the desired
expression of $f_t$. Conversely, if the series in the right hand
side of \eqref{E:solution} converges locally uniformly, then a
straightforward computation shows that the resulting function $f_t$
determined by \eqref{E:solution} dose form a holomorphic solution of
equation \eqref{E:PDE}.
\end{proof}

\end{document}